\documentclass{amsart}
\usepackage{enumerate}
 \usepackage{upref}
\usepackage[hidelinks]{hyperref}

\newtheorem{theorem}{Theorem}[section]
\newtheorem{lemma}[theorem]{Lemma}
\newtheorem{corollary}[theorem]{Corollary}
\newtheorem{proposition}[theorem]{Proposition}

\theoremstyle{definition}

\theoremstyle{remark}
\newtheorem{remark}[theorem]{Remark}

\numberwithin{equation}{section}
\newcommand{\norm}[1]{\left\lVert #1 \right\rVert}
\newcommand{\abs}[1]{\left\lvert #1 \right\rvert}

\def\IR{\mathbb{R}}

\def\EE{\mathcal{E}}
\def\GG{\mathcal{G}}
\def\FF{\mathcal{F}}

\def\d{\mathrm{d}}

\def\cF{\mathcal{F}}
\def\bL{\mathbf{L}}
\def\bS{\mathbf{S}}
\def\N{\mathbb{N}}
\def\T{\mathbb{T}}
\def\Cap{\text{Cap}}
\def\loc{\text{loc}}

\DeclareMathOperator{\supp}{supp}

\begin{document}
	
	\title{Regular subspaces of symmetric stable processes}
	
	
	\author{Dongjian Qian}
	\address{School of Mathematical Sciences, Fudan University, Shanghai 200433, People's Republic of China}
	\curraddr{}
	\email{djqian22@m.fudan.edu.cn}
	\thanks{}
	
	\author{Jiangang Ying}
	\address{School of Mathematical Sciences, Fudan University, Shanghai 200433, People's Republic of China}
	\curraddr{}
	\email{jgying@fudan.edu.cn}
	\thanks{The second author was supported by NSFC grant No. 11871162}
	
	\author{Yushu Zheng}
	\address{Shanghai Center for Mathematical Sciences, Fudan University, Shanghai 200433, People's Republic of China}
	\curraddr{}
	\email{yszheng666@gmail.com}
	\thanks{}
	
	
	
	\date{}
	
	\dedicatory{}
	
	\begin{abstract}
		Roughly speaking, regular subspaces are regular Dirichlet forms that inherit the original forms with smaller domains. In this paper, regular subspaces of 1-dim symmetric $\alpha$-stable processes are considered. The main result
		is that 
		it admits proper regular subspaces if and only if $\alpha\in [1,2]$.
		Moreover, for $\alpha\in(1,2)$, the characterization of the regular subspaces is given. General 1-dim symmetric L\'evy processes will also be investigated. It will be shown that whether it has proper regular subspaces is closely related to whether its sample paths have finite variation.
	\end{abstract}
	
	\maketitle

	\section{Introduction}
	The theory
	of Dirichlet space was formulated by A. Beurling and J. Deny \cite{beurling1959dirichlet} in 1959, which is
	an axiomatic extension of classical
	Dirichlet integrals in the direction of Markovian
	semigroups,
	and related to symmetric Markov processes naturally
	in terms of the regularity condition presented by M. Fukushima
	\cite{fukushima1971regular} in 1971, which paves a new approach
	to obtain Markov processes and has attracted
	a lot of attention since then.
	Let $(\mathcal{E},\mathcal{F})$ be a Dirichlet form on $L^2(E,m)$, where $E$ is a nice topological space and $m$ is a Radon measure. These notations will not be detailed since we mainly consider $L^2(\mathbb{R},\d x)$ in this paper. $(\mathcal{E},\mathcal{F})$ is called regular if $C_c\cap
	\mathcal{F}$ is dense in $\mathcal{F}$ under the norm $\mathcal{E}_1^{1/2}:=(\mathcal{E}+(\cdot,\cdot)_{L_2})^{1/2}$ and dense in $C_c$ under the uniform norm, where $C_c$ represents the space of continuous functions on $E$ with compact support. It was proved by Fukushima \cite{fukushima1971regular} that a regular Dirichlet form is always associated with a Hunt process and for details, refer to \cite{fukushima1980dirichlet,fukushima2010dirichlet}.
	
	In this paper, we shall consider the problem regarding regular subspaces
	of Dirichlet forms.
	Precisely, let $(\mathcal{E}^1,\mathcal{F}^1)$ and $(\mathcal{E}^2,\mathcal{F}^2)$ be two Dirichlet forms on $L^2(E,m)$. We call $(\mathcal{E}^2,\mathcal{F}^2)$ a (Dirichlet) subspace of $(\mathcal{E}^1,\mathcal{F}^1)$, and an (Dirichlet) extension conversely, if
	$$
	\mathcal{F}^2\subset\mathcal{F}^1;\, \mathcal{E}^2(u,v)=\mathcal{E}^1(u,v),\, \forall\, u,v\in \mathcal{F}^2.
	$$
	When both $(\EE^1,\FF^1)$ and $(\EE^2,\FF^2)$ are regular, we call
	$(\mathcal{E}^2,\mathcal{F}^2)$ a regular  subspace of $(\mathcal{E}^1,\mathcal{F}^1)$,
	and regular  extension conversely.
	When $\mathcal{F}^2\neq\mathcal{F}^1$, we call $\mathcal{F}^2$ a proper regular subspace. This concept is closely related to Silverstein subspaces and the readers may refer to \cite{he2022silverstein} for details. There are three basic questions we are interested in
	\begin{enumerate}
		\item Does a regular Dirichlet form have proper regular subspaces/extensions?
		\item How to characterize all regular subspaces/extensions if it does?
		\item How is the process associated with a regular subspace/extension related to
		the process associated with the original form?
	\end{enumerate}
	
	Although the problem is not hard to be understood, there have been not many results yet. 
	The first
	result in this direction was presented in 2005 \cite{fang2010regular}, in which
	the authors, including M. Fukushima and the second author of the current paper, 
	investigated the problem of regular subspaces for the Dirichlet form associated with the 1-dim Brownian motion, namely the Sobolev space $H^1$ with the form
	$$
	\frac{1}{2}\mathcal{D}(u,v)=\frac{1}{2}\int_{\mathbb{R}}u^\prime(x) v^\prime(x) \d x,\quad u,v \in H^1.
	$$
	For any strictly increasing and continuous function $s$,
	called a scale function, on $\IR$,
	define a symmetric Dirichlet form $(\mathcal{E}^{(s)},\mathcal{F}^{(s)})$ on $L^2(\mathbb{R})$ as
	\begin{align*}
		\left\{\begin{aligned}
			&\mathcal{F}^{(s)}=\Big\{u\in L^2(\mathbb{R}):u\ll s, \int_ {\mathbb{R}}\Big(\frac{\d u}{\d s}\Big)^2 \d s<\infty\Big\};\\
			&\mathcal{E}^{(s)}(u,v):=\frac{1}{2}\int_ {\mathbb{R}}\frac{\d u}{\d s}\frac{\d v}{\d s}\d s.
		\end{aligned}\right.
	\end{align*}
	where $u\ll s$ means $u$ is absolutely continuous with respect to $s$.
	It was shown in \cite{fang2010regular}  that $(\mathcal{E},\mathcal{F})$ is a regular subspace of $(\frac{1}{2}\mathcal{D},H^1$) if and only if
	$(\EE,\FF)=(\EE^{(s)},\FF^{(s)})$ with $s$ satisfying
	$$
	\d s\ll \d x, \text{and}\ s^\prime =0 \text{ or } 1 \  a.e.
	$$
	Hence this gives a complete characterization of regular subspaces for 1-dim Brownian motion. 
	Actually, the question regarding one-dimensional diffusions has been completely
	solved, including the characterization of irreducible regular extensions of $(\frac{1}{2}\mathcal{D},H^1)$, 
	in \cite{fang2010dirichlet}. 
	General regular extensions of Brownian motion and 1-dim diffusions can be
	characterized with the help of effective intervals in
	\cite{li2019symmetric,li2020effective}.
	It is seen that the process associated with the regular subspace of Brownian motion
	is a diffusion with a scale function of Cantor-type, which can not be described
	by an SDE. Hence
	the problem is not only fundamental in the theory of
	Dirichlet forms, but also important probabilistically
	since every regular subspace or extension gives birth to
	a symmetric Markov process, which may not be known to us before.
	
	However, for other processes, few results have been obtained.
	Some results were given in \cite{li2015regular}. 
	Brownian motion with dimension $d>1$ admits proper regular subspaces,
	though a complete characterization is still open.
	The Dirichlet form
	associated with a symmetric step process 
	(for example, a symmetric pure jump L\'evy process with finite
	L\'evy measure) admits no proper
	regular subspaces.
	It was also proved there that 
	the jump measure and killing measure of a Dirichlet form will be inherited by its regular subspaces.
	
	It is then natural to ask if a pure jump Dirichlet form
	with infinite jump measure, including symmetric stable processes,
	admit 
	proper regular subspaces. Actually in this case 
	an example admitting proper regular subspace was given
	%
	in terms of the trace space of Brownian motion \cite{li2019regular}.
	
	In this paper, we focus on 1-dim symmetric $\alpha$-stable processes ($0<\alpha< 2$). The associated Dirichlet forms are the fractional Sobolev space
	\begin{align*}
		H^{\alpha/2}=\left\{f \in L^{2}(\mathbb{R}):\int_{\mathbb{R}}|\widehat{f}(\xi)|^2(1+|\xi|^{\alpha})\d\xi <\infty \right\},
	\end{align*}
	where $\widehat{f}(\xi)$ is the Fourier transform of $f$, 
	equipped with the form
	\begin{align*} \mathcal{E}(f,g)=\EE^{\alpha/2}(f,g)=\int_{\mathbb{R}}\widehat{f}(\xi)\bar{\widehat{g}}(\xi)|\xi|^{\alpha}\d\xi.
	\end{align*}
	It's shown in \cite{fukushima2010dirichlet} that it can also be expressed as
	\begin{align}\label{def}
		\left\{\begin{aligned}
			&H^{\alpha/2}=\left\{f \in L^{2}(\mathbb{R}):\int_{\mathbb{R}^2}\dfrac{(f(x)-f(y))^2}{|x-y|^{1+\alpha}}\d x\d y<\infty \right\};\\
			&\mathcal{E}(f,g)=
			\frac{C(\alpha)}{2}\int_{\mathbb{R}^2}\dfrac{(f(x)-f(y))(g(x)-g(y))}{|x-y|^{1+\alpha}}\d x\d y,
		\end{aligned}\right.
	\end{align}
	where $C(\alpha)$ is a positive constant. The case $\alpha=2$, corresponding to Brownian motion, may be included if necessary.
	The main result of the paper is stated as follows.
	
	\begin{theorem}\label{mainthm}
		The 1-dim symmetric $\alpha$-stable process has proper regular subspaces if and only if $\alpha\in[1,2]$.
	\end{theorem}
	Moreover, we give a characterization of the regular subspaces for $1<\alpha\le 2$. Recall that a scale function refers to a continuous and strictly increasing function on $\IR$.  For any scale function $s$, we define
	\begin{align*}
		\mathbf{L}_s&:=\big\{f:s(\IR)\rightarrow \IR:f\text{ is Lipschitz and }\supp(f)\subset s(\IR)\big\},\\
		\mathcal{F}_s&:=\big\{f\circ s:f\in \bL_s\big\}.
	\end{align*}
	We will mainly focus on the scale function of the following type.
	Denote by $H^{\alpha/2}_{\loc}$ the space of functions locally in $H^{\alpha/2}$, i.e. $f\in H^{\alpha/2}_{\loc}$ if for any compact set $K$, there exists $g\in H^{\alpha/2}$ with $g=f$ a.e. on $K$. 
	We set
	\[\bS^{\alpha/2}_{\loc}:=\big\{s\in H^{\alpha/2}_{\loc}:s\text{ is a scale function}\big\}.\]

	\begin{theorem}\label{char1}
		Assume $1 <\alpha\le 2$ and $s\in \mathbf{S}^{\alpha/2}_{\loc}$. Then we have $\cF_s\subset H^{\alpha/2}$ and $\overline{\cF_s}=\overline{\cF_s}^{\EE^{\alpha/2}_1}$ is a regular subspace of $H^{\alpha/2}$. Conversely,  any regular subspace of $H^{\alpha/2}$ equals $\overline{\cF_s}$ for some $s\in\mathbf{S}^{\alpha/2}_{\loc}$.
	\end{theorem}
	
	Next, we further give several necessary and sufficient conditions on $s$ for $\overline{\cF_s}$ to be a proper subspace. For $\alpha=2$, we can readily tell whether it's proper by considering the Lebesgue measure of $\{\d s/\d x=0\}$, as is done by \cite{fang2006algebraic}.
	For $\alpha\in (0,2)$, set $\alpha^*=2-\alpha$. Denote by $\Cap^{\alpha^*/2}$ the capacity relative to $(H^{\alpha^*/2},\EE^{\alpha^*/2})$. For a Borel set $E$, we say $\Cap^{\alpha^*/2}$ is concentracted on $E$ if for any open set $U\subset \IR$, $\Cap^{\alpha^*/2}(E\cap U)=\Cap^{\alpha^*/2}(U)$. We use $\d s$ to represent the Radon measure induced by a scale function $s$ and write $\langle f,\d s\rangle=\int_{\IR}f\,\d s$ for a Borel function $f$. For $f\in H^{\alpha^*/2}$, $\widetilde{f}$ represents its quasi continuous modification.
	\begin{theorem}\label{char2}
		For $1< \alpha< 2$ and $s\in \mathbf{S}^{\alpha/2}_{\loc}$, $\d s$ is a smooth measure and
		the following are equivalent:
		\begin{enumerate}[(a)]
			\item\label{proper1} $\overline{\cF_s}$ is proper;
			\item\label{proper2} $\Cap^{\alpha^*/2}$ is not concentrated on the quasi support of $\d s$;
			\item\label{proper3} 
			there exists an open set $E$ such that $\d s$ is concentrated on $E$ (i.e. $\langle 1_{E^c},\d s\rangle=0$) and $\Cap^{\alpha^*/2}$ is not concentrated on $E$;
			\item\label{proper4} 
			there exists a nonzero $f\in H^{\alpha^*/2}$ such that $\langle \widetilde{f},\d s\rangle=0$,
		\end{enumerate}
		where the quasi notions `smooth', `quasi support', `q.e.', and `$\widetilde{f}$' are all relative to $(H^{\alpha^*/2},\EE^{\alpha^*/2})$ and a function is nonzero if it is not identically zero.
	\end{theorem}
	We point out that except for the second assertion in Theorem \ref{char1}, all other conclusions in Theorem \ref{char1} and \ref{char2} also hold for $\alpha=1$. 
	

	The paper is organized as follows. In Section \ref{presection}, we prepare a lemma for the later proof. Section \ref{case1}$\sim$\ref{ifpart} is devoted to the proof of Theorem \ref{mainthm}$\sim$\ref{char2}. In Section \ref{generalsec}, we extend results to general symmetric L\'evy processes.
	
	The proofs in the paper use quite different methods from the previous ones about the Brownian motion. It may shed a light on the existence of regular subspaces of general pure jump processes, though the methods are highly analytical.
	
	\subsection*{Notation}
	We will use the following notation throughout the paper. For $-\infty\le a<b\le \infty$, we denote by
	\begin{itemize}
		\item $C_c(a,b)$ the space of continuous functions on $(a,b)$ with compact support;
		\item $C_\infty(a,b)$ the space of continuous functions on $(a,b)$ which vanish at $a$ and $b$;
		\item $C^\infty_c(a,b)$ the space of smooth functions on $(a,b)$ with compact support.
	\end{itemize}
	In particular, `$(a,b)$' in the above notation is omitted when $(a,b)=\IR$.
	
	\section{A preliminary lemma}\label{presection}
	In this section, we give a preliminary lemma (Lemma \ref{normcontract}) which will be frequently used in the sequel. First, we present a simple generalization of a
	result in Sobolev spaces. 
	\begin{proposition}\label{hardy}
		Let $\alpha\in (0,1)\cup(1,2]$ and $(a,b)$ be a finite open interval. For $f\in L^2(a,b)$ (if $\alpha\in (0,1)$) or $f\in C_\infty(a ,b)$ (if $\alpha\in(1,2)$),
		\begin{align}\label{hardy1}	\begin{aligned}
				\int_{(a,b)\times\mathbb{R}\backslash (a,b) } \dfrac{f(x)^2}{|x-y|^{1+\alpha}}\d x \d y
				&=\frac{1}{\alpha}\int_{(a,b)} f(x)^2 \left(\dfrac{1}{|x-a|^{\alpha}}+\dfrac{1}{|x-b|^{\alpha}}\right)\d x\\
				&\leq c_\alpha \left(\int_{(a,b)^2} \dfrac{(f(x)-f(y))^2}{|x-y|^{1+\alpha}}\d x\d y+\norm{f}^2_{L^2(a,b)}\right),
			\end{aligned}
		\end{align}
		where $c_\alpha$ is a constant depending only on 
		$\alpha$.
	\end{proposition}
	\begin{proof}
		It suffices to consider $f$ such that the right-hand side of \eqref{hardy1} is finite.
		In \cite[Lemma 3.32]{mclean2000strongly}, the analogous result for $f\in C^\infty_c(a,b)$ is given. For $\alpha\in (0,1)$, the generalization is immediate from the fact that $C^\infty_c(a,b)$ is dense in 
		\begin{align}\label{Hdef}
			H^{\alpha/2}(a,b)=\left\{f\in L^2(a,b):\int_{(a,b)^2} \dfrac{(f(x)-f(y))^2}{|x-y|^{1+\alpha}}\d x\d y<\infty\right\}.
		\end{align}
		For $\alpha\in (1,2)$, it is known (Cf. \cite[Theorem 3.40]{mclean2000strongly}) that 
		\[\overline{C^\infty_c(a,b)}^{\EE^{\alpha/2}_1}= H^{\alpha/2}(a,b)\cap C_\infty(a,b),\] which implies the conclusion. 
	\end{proof}
	For two non-negative functions $f$ and $g$ in $C_c$, we call $f$ an erased function of $g$ if $f\le g$ and $f$ is constant on each connected component of $\{f<g\}$. Intuitively, $f$ is obtained by `erasing' some upwards excursions of $g$. Henceforth we shall simply write $\EE=\EE^{\alpha/2}$ whenever the parameter `$\alpha$' is clear from the context.
	
	\begin{lemma}\label{normcontract}
		Let $\alpha\in (0,1)\cup(1,2]$, $\mathcal{G}$ be a regular subspace of $H^{\alpha/2}$, and $f$ and $g$ be non-negative functions in $C_c$.  If $g\in \mathcal{G}$ and $f$ is an erased function of $g$, then $f\in \mathcal{G}$ and $\norm{f}_{\mathcal{E}_1}\leq 6(1+c_\alpha) \norm{g}_{\mathcal{E}_1}$.
	\end{lemma}
	\begin{proof}
		The proof is divided into two steps. 
		We first show $f\in H^{\alpha/2}$ and $\norm{f}_{\EE_1}\le 6(1+c_\alpha)\norm{g}_{\EE_1}$, and then prove $f\in \mathcal{G}$.
		
		\textbf{Step 1}.
		Denote $\{f<g\}=\bigcup_{i\ge 1}(a_i,b_i)$, where $(a_i,b_i)$ are disjoint intervals. Here for simplicity, we also see $\emptyset$ as an interval, which ensures that the right-hand side can always be taken as a countable union. For $n\ge 1$, let
		$$
		f_n(x)=\begin{cases}
			f(x),\ & x\in \bigcup_{1\leq i\leq n}(a_i,b_i) ;\\
			g(x),\ & x\in \mathbb{R}\backslash\bigcup_{0\leq i\leq n}(a_i,b_i).
		\end{cases}
		$$
		It suffices to prove $\norm{f_n}_{\EE_1}\le 6(1+c_\alpha)\norm{g}_{\EE_1}$. In fact, once it is proved, we have the $\EE_1$-norms of $f_n$ are uniformly bounded.  It follows that the Ces\`aro means of a subsequence of $f_n$ are $\EE_1$-convergent. On the other hand, $f_n$ converges uniformly to $f$. Thus, $f\in H^{\alpha/2}$ and $\norm{f}_{\EE_1}\le 6(1+c_\alpha)\norm{g}_{\EE_1}$.
		
		Set $D=\bigcup_{1\leq i\leq n}(a_i,b_i)$. We will divide the domain of the integration in \eqref{def} and compare the integrations associated with $f_n$ and $g$ term by term.  
		\begin{enumerate}[(1)]
			\item\label{integration1} Since $f_n$ is constant on each $(a_i,b_i)$, $1\le i\le n$, we have
			\[\int_{\bigcup_{1\leq i\leq n}(a_i,b_i)^2}\dfrac{(f_n(x)-f_n(y))^2}{|x-y|^{1+\alpha}}\d x\d y=0.\]
			\item\label{integration2} Since $f_n=g$ outside $D$, it is plain that
			\[\int_{(D^c)^2}\dfrac{(f_n(x)-f_n(y))^2}{|x-y|^{1+\alpha}}\d x\d y=\int_{(D^c)^2}\dfrac{(g(x)-g(y))^2}{|x-y|^{1+\alpha}}\d x\d y.\]
			\item\label{integration3} For $1\le i\le n$,  
			\begin{align*}
				&\int_{(a_i,b_i)\times D^c } \dfrac{(f_n(x)-f_n(y))^2}{|x-y|^{1+\alpha}}\d x \d y\\
				\leq &2\int_{(a_i,b_i)\times D^c} \left[\dfrac{(g(x)-f_n(x))^2}{|x-y|^{1+\alpha}}+ \dfrac{(g(x)-g(y))^2}{|x-y|^{1+\alpha}}\right]\d x \d y.
			\end{align*}
			\item\label{integration4} For $1\le i,j\le n$ and $i\neq j$,
			\begin{align*}
				&\int_{(a_i,b_i)\times (a_j,b_j) } \dfrac{(f_n(x)-f_n(y))^2}{|x-y|^{1+\alpha}}\d x \d y\\
				\le &3\int_{(a_i,b_i)\times (a_j,b_j) } \left[\dfrac{(f_n(x)-g(x))^2}{|x-y|^{1+\alpha}}+ \dfrac{(f_n(y)-g(y))^2}{|x-y|^{1+\alpha}}+ \dfrac{(g(x)-g(y))^2}{|x-y|^{1+\alpha}}\right]\d x \d y.
			\end{align*}
		\end{enumerate}
		Summing the parts \eqref{integration3} and \eqref{integration4}, we have
		\begin{align*}
			&\sum_{1\leq i\leq n}\int_{(a_i,b_i)\times(a_i,b_i)^c } \dfrac{(f_n(x)-f_n(y))^2}{|x-y|^{1+\alpha}}\d x \d y\\
			\le&6\sum_{1\leq i\leq n}\int_{(a_i,b_i)\times(a_i,b_i)^c }\left[\dfrac{(g(x)-f_n(x))^2}{|x-y|^{1+\alpha}}+ \dfrac{(g(x)-g(y))^2}{|x-y|^{1+\alpha}}\right]\d x \d y\\
			\le& 6(1+c_\alpha)\left(\sum_{1\leq i\leq n}\int_{(a_i,b_i)\times(a_i,b_i)^c } \dfrac{(g(x)-g(y))^2}{|x-y|^{1+\alpha}}\d x \d y+\norm{g-f_n}^2_{L^2}\right),
		\end{align*}
		where the last inequality follows from Proposition \ref{hardy}. Further adding \eqref{integration1} , \eqref{integration2}, and $\norm{f_n}^2_{L^2}$, we get $\norm{f_n}_{\EE_1}\le 6(1+c_\alpha)\norm{g}_{\EE_1}$.
		
		\textbf{Step 2}.
		Denote $h=g-f$. It suffices to prove $h\in \mathcal{G}$. Since $f\in H^{\alpha/2}$, so is $h$. Then we can take $\varepsilon$-contraction on $h$,
		$$h_{\varepsilon}:=h-(-\varepsilon)\vee h\wedge \varepsilon.$$ 
		It is known that $h_{\varepsilon} \stackrel{\mathcal{E}_1}{\rightarrow}h$ as $\varepsilon \downarrow 0$
		and hence
		it is enough to prove $h_{\varepsilon}\in\GG$ for any $\varepsilon>0$.
		
		Recall that $\{h>0\}=\{f<g\}=\bigcup_{i\geq 1}(a_i,b_i)$.
		Since $h$ is uniformly continuous, there exist at most finite $i$ such that $$(a_i,b_i)\cap\{h>\varepsilon \} \neq \emptyset .$$ Without loss of generality, we assume  $$h_{\varepsilon}=\left( \sum_{1\leq i\leq N}h1_{(a_i,b_i)} \right)_{\varepsilon}= \sum_{1\leq i\leq N}\left(h1_{(a_i,b_i)}\right)_{\varepsilon} $$ for some integer $N$. Then we need only to prove that for each $i$, $(h1_{(a_i,b_i)})_\varepsilon\in\mathcal{G}$. Remember that $f$ is constant on $[a_i,b_i]$, i.e.
		$f(x)=d_i$ for $x\in [a_i,b_i]$ for some constant $d_i$. Hence, $h1_{(a_i,b_i)}=(g-d_i)1_{(a_i,b_i)}$.
		Since $g(a_i)=g(b_i)=d_i$, it follows from continuity that there exists $\delta>0$ such that for $x\in [a_i,a_i+\delta]\cup [b_i-\delta,b_i]$, $$0\le g(x)-d_i \leq \varepsilon.$$
		Due to the regularity of $\GG$, we can take a continuous function $g_i \in \mathcal{G}$, such that $0\le g_i\le 1$ and
		$$
		g_i(x)=\left\{
		\begin{array}{ll}
			1,\ &x\in (a_i+\delta,b_i-\delta];\\
			0,\ &x\in (-\infty, a_i]\cup [b_i,\infty).
		\end{array} \right.
		$$
		It follows easily that $(g-d_i)g_i\in \mathcal{G}$.
		The difference between $(g-d_i)1_{(a_i,b_i)}$ and $(g-d_i)g_i$
		is located on $(a_i,a_i+\delta]\cup(b_i-\delta, b_i)$  and both
		will be erased by $\varepsilon$-contraction. 
		Therefore it holds that $$(h1_{(a_i,b_i)})_{\varepsilon}=
		((g-d_i)1_{(a_i,b_i)})_{\varepsilon}=((g-d_i)g_i)_{\varepsilon}\in\GG.$$
		That completes the proof.
	\end{proof}

	\section{Proof of the `only if' part of Theorem \ref{mainthm}}\label{case1}
	We shall prove that for $0<\alpha<1$, a symmetric $\alpha$-stable process has no proper regular subspaces. If $H^{\alpha/2}$ has a regular subspace, we denote it by $\mathcal{G}$ and deduce that in fact $\mathcal{G}=H^{\alpha/2}$. Before the proof,
	let us prepare a few lemmas. 
	The first lemma is an easy but key observation.
	
	\begin{lemma} \label{lemma11}
		For $a<b$, $1_{(a,b]}\in H^{\alpha/2}$ if and only if $\alpha<1$.
	\end{lemma}
	
	\begin{proof} This is a direct computation. When $f=1_{(a,b]}$ and $\alpha\in(0,1)$,
		\begin{align*}&\int_{\IR^2}\dfrac{(f(x)-f(y))^2} {\abs{x-y}^{\alpha +1}}\d x\d y\\
			=&2\left(\int_{x\in(a,b],y>b}+\int_{x\in (a,b],y\le a}\right)\dfrac{1} {\abs{x-y}^{\alpha+1}}\d x\d y
			\\
			=&\dfrac{4}{\alpha(1-\alpha)}(b-a)^{1-\alpha}.
		\end{align*}
		When $\alpha\ge 1$, this integral diverges.
	\end{proof}
	Actually we don't need to distinguish $1_{(a,b]}$ and $1_{[a,b]}$.
	Denote by $$S_0=\left\{\sum_{i=1}^K c_i1_{(a_i,b_i]}:K\ \text{is an integer},~ c_i\in \mathbb{R},~ (a_i,b_i]\text{ are disjoint intervals}\right\}$$ the set of step functions. It is easily seen that $S_0\subset H^{\alpha/2} $ when $0<\alpha<1$.
	
	\begin{lemma}\label{S0}
		$S_0$ is dense in $H^{\alpha/2}$.
	\end{lemma}
	\begin{proof}
		It is well-known that
		$C^{\infty}_c$ is dense in $H^{\alpha/2}$. Hence it suffices to prove for any $f \in C^{\infty}_c$, we can find a sequence $\{f_n\}\subset S_0$ such that $f_n \stackrel{\mathcal{E}_1}{\longrightarrow}f$. Assume that $f$ is supported on $[-N,N]$ and $|f^{\prime}|<c$ for some positive constants $N$ and $c$.
		
		Take $f_n(x)=f(\frac{i}{2^n})$, whenever $\frac{i}{2^n}< x\leq \frac{i+1}{2^n}$
		for any integer $i$. It obvious that $|f(x)-f_n(x)|\le c/2^n$ for
		any $x\in\IR$,
		$f_n\in S_0$ and $f_n$ converges to $f$ both pointwise and in $L^2$. We now aim to prove $f_n$ converges to $f$ in $\mathcal{E}$-norm.
		In the following proof, $c_i$ are all positive constants.
		\begin{small}
			\begin{align*}
				&\int_{\mathbb{R}^2}\dfrac{((f-f_n)(x)-(f-f_n)(y))^2}{|x-y|^{1+\alpha}}dxdy\\
				=&
				\left(2\int_{[-N,N]\times [-N,N]^c }
				+\int_{[-N,N]\times [-N,N]  }\right)\dfrac{((f-f_n)(x)-(f-f_n)(y))^2}{|x-y|^{1+\alpha}}\d x\d y\\
				=&
				\bigg(2\int_{[-N,N]\times [-N,N]^c }
				+2\sum_{-N 2^n\leq i<j < N 2^n} \int_{[\frac{i}{2^n},\frac{i+1}{2^n}]\times [\frac{j}{2^n},\frac{j+1}{2^n}]  }+\sum_{-N 2^n\leq i < N 2^n} \int_{[\frac{i}{2^n},\frac{i+1}{2^n}]\times [\frac{i}{2^n},\frac{i+1}{2^n}]  }\bigg)
				\\&\dfrac{((f-f_n)(x)-(f-f_n)(y))^2}{|x-y|^{1+\alpha}}\d x\d y=:I_1+I_2+I_3.
			\end{align*}
		\end{small}It can be verified that
			\begin{align*}
					I_1&\le 2\int_{[-N,N]\times \mathbb{R}\backslash [-N,N] }\dfrac{(c\frac{1}{2^n})^2}{|x-y|^{1+\alpha}}\d x\d y\le c_12^{-2n};\\
					I_2&\le 2\sum_{-N 2^n\leq i<j < N 2^n} \int_{[\frac{i}{2^n},\frac{i+1}{2^n}]\times [\frac{j}{2^n},\frac{j+1}{2^n}]  }\dfrac{(2c\frac{1}{2^n})^2}{|x-y|^{1+\alpha}}\d x\d y\\
					&\le c_2\sum_{-N 2^n\leq i<j \leq N 2^n} 2^{n(\alpha-3)}\le  4c_2N^22^{n(\alpha-1)};\\
					I_3&\le \sum_{-N 2^n\leq i < N 2^n} \int_{[\frac{i}{2^n},\frac{i+1}{2^n}]\times [\frac{i}{2^n},\frac{i+1}{2^n}]  }\dfrac{(c(x-y))^2}{|x-y|^{1+\alpha}}\d x\d y \\
					&\le c_3  \sum_{-N\times 2^n\leq i \leq N\times 2^n} 2^{n(\alpha-3)}\le 2c_3N2^{n(\alpha-2)}, 
				\end{align*}
				where for $I_3$, when $x,y\in [\frac{i}{2^n},\frac{i+1}{2^n}]$, we use the estimate $|(f-f_n)(x)-(f-f_n)(y)|=|f(x)-f(y)|\le c|x-y|$.
			Therefore  $I_1+I_2+I_3\le c_42^{n(\alpha-1)}$,
			which converges to 0 as $n\rightarrow \infty$.
		\end{proof}
		
		By this lemma, to prove that $\mathcal{G}=H^{\alpha/2}$, we only need to prove
		any indicator function $1_{[a,b]} \in \mathcal{G}$. To do this, we
		introduce a type of functions close to indicator functions.
		Fix $a<b$ from now and assume that 
		a continuous function $h_{\rho}$
		satisfies the following condition, 
		$$
		h_{\rho}(x)=\begin{cases}
			1 ,&  x\in[a,b];\\
			0  ,&  x\in(-\infty, a-\rho]\cup [b+\rho,\infty);\\
			\text{increasing} ,\ &  x\in(a-\rho,a];\\
			\text{decreasing} ,\ &  x\in(b, b+\rho).
		\end{cases}
		$$
		with some $\rho>0$. 
		Notice that when $x \in (a-\rho, a)\cup (b, b+\rho)$, $h_\rho(x)$ is not specifically defined, as long as it is monotone. Denote by $M_\rho$  the collection of all such functions. Now it suffices to prove the following lemma
		\begin{lemma}\label{hn}
			For any $\rho>0$, it holds that
			\begin{enumerate}[(i)]
				\item\label{ladder1}  $M_\rho\cap \mathcal{G}
				\neq \emptyset$;
				\item\label{ladder2} $\big\{\norm{f}_{\EE_1}:f\in M_\rho\big\}$ are uniformly bounded. 
			\end{enumerate}
		\end{lemma}
		If the lemma is proved, we take a sequence $\{\rho_n\}$ decreasing to 0, say $\rho_n=1/n$. Due to \eqref{ladder1}, for each $n$,
		there is a function $h_{\rho_n}\in M_{\rho_n}\cap \mathcal{G}$. It follows from \eqref{ladder2} that the norms $\big\{\Vert h_{\rho_n} \Vert_{\mathcal{E}_1}:n\ge 1\big\}$ is bounded. Also note that $h_{\rho_n}$ converges pointwise to $1_{[a,b]}$. Hence the standard arguments yield that $1_{[a,b]}\in\mathcal{G}$.
		\begin{proof}
			For \eqref{ladder1}, since $\mathcal{G}$ is regular, there exists $g \in \mathcal{G}\cap C_c$ such that
			$0\le g\le 1$ and
			\begin{align*}
				g(x)=\begin{cases}
					1,\ & x\in[a,b];\\
					0,\ & x\in (a-\rho, b+\rho)^c.
				\end{cases}
			\end{align*}
			Then we define
			\begin{align}\label{skorokhod}
				g^*(x):=\left\{
				\begin{array}{ll}
					1,\ &  x\in[a,b];\\
					0,\ &  x\in (a-\rho, b+\rho)^c
					;\\
						\inf\{g(t):t\in [x,a]\},\ &  x\in(a-\rho,a);\\
							\inf\{g(t):t\in [b,x]\},\ &  x\in(b, b+\rho).
						\end{array} \right. 
					\end{align}
					The construction is just like the Skorokhod decomposition and shares the same good properties: $g^*\in M_\rho$ and $g^*$ is an erased function of $g$. By Lemma \ref{normcontract}, we have $g^*\in \GG$.
					
					For \eqref{ladder2}, we only need to prove that 
					$$\sup_{f\in M_\rho}\int_{|\xi|\geq 1}|\hat{f}(\xi)|^2|\xi|^{\alpha}d\xi<\infty.$$ Using the integration by parts formula and noticing that any $f\in M_\rho$ is of finite variation, we have for $\abs{\xi}\ge 1$,
					\begin{align*}
						\hat{f}(\xi)&=\frac{1}{\sqrt{2\pi}}\int_{\mathbb{R}}f(x)(\cos(x\xi)+i\sin(x\xi))\d x\\
						&=-\frac{1}{\sqrt{2\pi}}\left(\int_{\mathbb{R}}\dfrac{\sin(x\xi)}{\xi}\d f(x)-i\int_{\mathbb{R}}\dfrac{\cos(x\xi)}{\xi}\d f(x)\right).
					\end{align*}
					Observe that the total variation of $f$ is $2$. Then it follows that\begin{equation}\label{keyinequality} |\hat{f}(\xi)|\leq {\dfrac{4}{\sqrt{2\pi}\abs{\xi}}}\le \dfrac{2}{\abs{\xi}}\, \text{ for }\abs{\xi}\ge 1.\end{equation}
					Therefore, for any $f\in M_\rho$,
					$$\int_{|\xi|\geq 1}|\hat{f}(\xi)|^2|\xi|^{\alpha}d\xi \leq \int_{|\xi|\geq 1} 4|\xi|^{\alpha-2}d\xi<\infty.$$
					That completes the proof of the lemma and also the `only if' part of Theorem \ref{mainthm}.   
				\end{proof}

							\section{Proof of Theorem \ref{char1}}\label{charsect}
							In this section, we will prove Theorem \ref{char1}, which gives a characterization of the regular subspaces of
							$\alpha$-stable process when $\alpha\in(1,2)$. The former part of Theorem \ref{char1} is straightforward. The first conclusion `$\cF_s\subset H^{\alpha/2}$' comes from  $s\in H^{\alpha/2}_{\loc}$ and the definition of $\mathbf{L}_s$. Moreover, $\mathcal{F}_s$ is a lattice separating points. Hence it is dense in $C_c$ under the norm $\norm{\cdot}_\infty$. Then it's standard to verify that $\overline{\mathcal{F}_s}$ is a regular subspace of $H^{\alpha/2}$.
							
							The latter part of \ref{char1} is more involved. Assume $\GG$ is a regular subspace of $H^{\alpha/2}$. We shall find $s\in \bS^{\alpha/2}_{\loc}$ such that $\GG=\overline{\cF_s}$.
							The key is to decompose a non-negative function in $C_c\cap \GG$ into a countable sum of ladder-like functions, that are functions in $C_c$ which first increase to a maximum, then decrease to $0$ and stay at $0$ afterwards. Based on this, the main idea of the following proof is to first take a countable set of functions in $C_c\cap \GG$ which is dense in $\GG$, and then use the ladder-like functions in the decompositions of these functions to construct the function $s$ we need.
							
							\begin{lemma}\label{decomposition}
								Let $1<\alpha\le 2$ and $\GG$ be a regular subspace of $H^{\alpha/2}$. For any non-negative function $f\in C_c\cap \GG$, there exists a 
								sequence of ladder-like functions $f^n\in\GG$, such that
								$\sum_n f^n$ converges in $\EE_1$-norm to $f$.
							\end{lemma}
							\begin{proof}
								First, we associate every non-negative function $f\in C_c$ with a ladder-like function as follows. Let $T=T^f$ be the minimum of the maximum points of $f$, i.e. $T=\inf\big\{x\in \IR:f(x)=\norm{f}_\infty\big\}$ . We consider the following  decomposition of $f$, which is similar to that in \eqref{skorokhod}.
								\begin{align*}
									f^*(x):=\left\{
									\begin{array}{ll}
										\inf\{f(t):t\in[x,T]\},\ &  x\in(-\infty, T],\\
										\inf\{f(t):t\in[T,x]\},\ &  x\in(T,\infty).
									\end{array} \right. 
								\end{align*}
								We call $f^*$ the ladder-like function associated with $f$. Moreover, we denote $\{f^*<f\}=\bigcup_{i\ge 1}(a_i,b_i)$, where $(a_i,b_i)$ are disjoint intervals. 
								Let $f_i:=(f-f^*)1_{(a_i,b_i)}$. Then
								\[f=f^*+\sum_i f_i,\]
								which is called the ladder decomposition of $f$. And $\{f_i:i\ge 1\}$ are called the ladder excursions of $f$.
								
								Now we proceed to consider the ladder decomposition of each $f_i$. Let $\{f_{ij}:j\ge 1\}$ be the ladder excursions of $f_i$. Recursively, suppose we have defined $f_{n_1n_2\cdots n_m}$ for $n_1n_2\cdots n_m=(n_1,n_2,\cdots,n_m)\in \N^m$, where $\N:=\{1,2,\cdots\}$. Then $\{f_{n_1n_2\cdots n_mj}:j\ge 1\}$ is defined to be the ladder excursions of $f_{n_1n_2\cdots n_m}$.
								We also denote $f_\emptyset=f$. In this way, we have defined $\big\{f_u:u\in \{\emptyset\}\cup\bigcup_{n\in \N}\N^n\big\}$. 
								
								Note that $\{\emptyset\}\cup\bigcup_{n\in \N}\N^n$ can be naturally viewed as the vertices set of a tree $\T=(V, E)$ where $n_1n_2\cdots n_m$ has children $\{n_1n_2\cdots n_mi:i\ge 1\}$. Let $\T_n=(V_n,E_n)$ be a sequence of finite connected subgraphs of $\T$, such that $\emptyset\in \T_n$ for all $n$ and $V_n\uparrow V$.
								In the following, we shall prove $\sum_{u\in V_n}f^*_u$ converges pointwise to $f$. Once it is proved, by further noting that $\sum_{u\in V_n}f^*_u$ are erased functions of $f$ and using Lemma \ref{normcontract}, we have $\sum_{u\in V_n}f^*_u\in \GG$ for all $n$ and the 
								$\EE_1$-norms of  $\sum_{u\in V_n}f^*_u$ are uniformly bounded. It follows that $\sum_{u\in V_n}f^*_u$ also converges to $f$ in $\EE_1$, which concludes the lemma.
								
								For the proof of pointwise convergence, we first present an easy observation. We can see from the construction of $f_u$ that $\sum_{u\in V_n}f^*_u$ is 
								obtained from $f$ by `erasing' the upwards excursions over the intervals $\supp(f_u)$,
								$u\in V\setminus V_n$. It follows that 
								\begin{align}\label{observation}
									\sum_{u\in V_n}f^*_u=f\text{ on }A_n:=\supp(f)\setminus \Big(\bigcup_{u\in V\setminus V_n}\supp(f_u)\Big).
								\end{align}
								Fix $x\in \supp(f)$. Then $x$ is either in finitely many $\supp(f_u)$ or in infinitely many of them. In the former case, it is direct from \eqref{observation} that  for $n$ sufficiently large, $\sum_{u\in V_n}f^*_u(x)=f(x)$. Henceforth we focus on the latter case. We call a sequence $\{u_n:n\ge 0\}$ of $V$ a ray in $\T$ if $u_0=\emptyset$ and for any $n\ge 0$, $u_n$ is the parent of $u_{n+1}$. In the latter case, we can always find a ray $\{u_n\}$ in $\T$, such that $x\in \supp(f_{u_n})$ for all $n$. 
								For convenience, we denote $T_n=T^{f_{u_n}}$ and $[l_n,r_n]=\supp(f_{u_n})$ for $n\ge 1$. From the construction of $f_{u_n}$, we see that $l_n,r_n,T_n\in A_m$ whenever $u_n\in V_m$. Combining with \eqref{observation}, we have $f(y)=\sum_{u\in V_m}f^*_u(y)$ for $y=l_n,r_n,T_n$. Hence,
								\[f(l_n)=f(r_n)\le \sum_{u\in V_m}f^*_u(x)\le f(T_n).\]
								Note that the inequality still holds if we replace $ \sum_{u\in V_m}f^*_u(x)$ with $f(x)$. So it remains to show $f(T_n)-f(l_n)=f(T_n)-f(r_n)$ goes to $0$. Back to the construction of $f_{u_n}$ again, it is seen that $[l_n,r_n]$ is decreasing and either $[l_{n+1},r_{n+1}]\subset (l_n,T_n)$ or $[l_{n+1},r_{n+1}]\subset (T_n,r_n)$. In particular, one of $(l_n,T_n)$ and $(T_n,r_n)$ is a subset of $[l_n,r_n]\setminus [l_{n+1},r_{n+1}]$. Thus, 
								\[(T_n-l_n)\wedge (r_n-T_n)\rightarrow 0.\]
								The conclusion we want then follows from the uniform continuity of $f$.
							\end{proof} 
							Now we start the construction of $s$. Since $H^{\alpha/2}$ is separable and $\GG$ is regular, we can find a countable set $\mathcal{B}\subset C_c\cap \GG$ which is dense in $\GG$. Consider $\mathcal{B}^\pm:=\{f^\pm:f\in \mathcal{B}\}$ the positive and negative parts of the functions in $\mathcal{B}$. For any $f\in \mathcal{B}^+\cup \mathcal{B}^-$, let $\{f^n:n\ge 1\}$ be the ladder-like functions satisfying the conclusion in Lemma \ref{decomposition}. For simplicity, we let $\{h_n:n\ge 1\}$ be an enumeration of $\{f^n:n\ge 1,\, f\in \mathcal{B}^+\cup \mathcal{B}^-\}$. Due to Lemma \ref{decomposition}, the linear span of $\{h_n:n\ge 1\}$ is dense in $\GG$. We denote $[l_n,r_n]=\supp(h_n)$ and $T_n=T^{h_n}$. 
							Define 
							\[h^l_n=h_n\vee 1_{[T_n,\infty]} \norm{h_n}_\infty,\quad h^r_n=h^l_n-h_n=1_{[T_n,\infty)}(\norm{h_n}_\infty-h_n).\] Intuitively, $h^l_n$ is the `left arm' of the ladder-like function $h_n$ and $h^r_n$ is a reflection of the `right arm' of $h_n$.
							
							Recall the definition of $\EE$ in \eqref{def}, whose domain can be extended to all measurable functions. For simplicity, we omit the coefficient $C(\alpha)/2$ in the definition henceforth. We shall first prove that $\EE\big(h^l_n,h^l_n\big)$ and  $\EE\big(h^r_n,h^r_n\big)$ are both finite. Since $h^r_n=h^l_n+h_n$, it suffices to prove $\EE\big(h^l_n,h^l_n\big)<\infty$. We have
							\begin{align*}
								&\EE\big(h^l_n,h^l_n\big)\\
								=&\int_{(-\infty,T_n)^2}\dfrac{(h_n(x)-h_n(y))^2}{\abs{x-y}^{1+\alpha}}\d x\d y+2\int_{(-\infty,T_n)\times(T_n,\infty)}\dfrac{(h_n(x)-\norm{h_n}_\infty)^2}{\abs{x-y}^{1+\alpha}}\d x\d y\\
								\le& \EE(h_n,h_n)+\dfrac{2\norm{h_n}_\infty^2}{\alpha(\alpha-1)(T_n-l_n)^{\alpha-1}}+2\int_{(l_n,T_n)\times(T_n,\infty)}\dfrac{(h_n(x)-\norm{h_n}_\infty)^2}{\abs{x-y}^{1+\alpha}}\d x\d y.
							\end{align*}
							It remains to bound the last term. For this, we take $\phi\in C^\infty_c$ such that $\phi=1$ on $[l_n,T_n]$ and consider $f=\phi\cdot(h_n-\norm{h_n}_\infty)$. It holds that
							\begin{align*}
								&\int_{(-\infty,T_n)^2}\dfrac{(f(x)-f(y))^2}{\abs{x-y}^{1+\alpha}}\d x\d y\\
								\le& 2 \norm{h_n}_\infty^2\EE(\phi,\phi)+2\norm{\phi}_\infty^2\int_{(-\infty,T_n)^2}\dfrac{(h_n(x)-h_n(y))^2}{\abs{x-y}^{1+\alpha}}\d x\d y<\infty.
							\end{align*}
							Note that $f\in C_\infty(-N,T_n)$ for $N$ large enough. So it follows from Lemma \ref{hardy} that
							\begin{align*}
								\int_{(l_n,T_n)\times(T_n,\infty)}\dfrac{(h_n(x)-\norm{h_n}_\infty)^2}{\abs{x-y}^{1+\alpha}}\d x\d y=\int_{(l_n,T_n)\times(T_n,\infty)}\dfrac{f(x)^2}{\abs{x-y}^{1+\alpha}}\d x\d y\\
								\le c_\alpha\left(\int_{(-\infty,T_n)^2}\dfrac{(f(x)-f(y))^2}{\abs{x-y}^{1+\alpha}}\d x\d y+\norm{f}^2_{L^2}\right)<\infty.
							\end{align*}
							
							Now we define $s$ as follows:
							\begin{equation*}
								s=\sum_{n\ge 1} 2^{-n} \dfrac{h^l_n +h^r_n}{\norm{h^l_n}_{\mathcal{E}}+\norm{h^r_n}_{\mathcal{E}}+\norm{h_n}_\infty}.
							\end{equation*}
							We see that $s$ is continuous. From the construction of $\{h_n\}$, it is easy to deduce that $s$ is strictly increasing. Moreover, since $\norm{s}_{\infty},\norm{s}_{\EE}\le 2$, we have for any $f\in C^\infty_c$, $f\cdot s\in H^{\alpha/2}$. Thus, $s\in \bS^{\alpha/2}_{\loc}$.
							
							The remaining part is devoted to the proof of $\overline{\cF_s}=\GG$. The `$\supset$' part is simple. In fact, for each $h_n$, it holds that
							\[\abs{h_n(x)-h_n(y)}\le C\abs{s(x)-s(y)},\]
							for some constant $C$. So $h_n\circ s^{-1}\in \bL_s$ and $h_n=h_n\circ s^{-1}\circ s\in \cF_s$. Recall that the linear span of $\{h_n\}$ is dense in $\GG$. Hence we get $\overline{\cF_s}\supset\GG$.
							
							For the other direction, it suffices to show every function in $\mathcal{F}_s$ can be approximated by functions in $\GG$. We only need to prove this for a certain class of functions. Fix $f\in \bL_s$. Let 
							\begin{equation*}
								s_k=\sum_{n=1}^k 2^{-n} \dfrac{h^l_n +h^r_n}{\norm{h^l_n}_{\mathcal{E}}+\norm{h^r_n}_{\mathcal{E}}+\norm{h_n}_\infty}.
							\end{equation*}
							Then $s_k$ converges pointwise to $s$, which implies $f\circ s_k$ converges pointwise to $f\circ s$. 
							Observe that 
							\[\abs{s_{k_1}(x)-s_{k_1}(y)}\le \abs{s_{k_2}(x)-s_{k_2}(y)}\text{ for any $k_1<k_2$}.\]
							So for $k$ large enough, $\supp(f\circ s_k)$ is compact and decreases in $k$. Further using Proposition \ref{hardy}, we have for above $k$, the 
							$\EE_1$-norms of  $f\circ s_k$ are uniformly bounded. Therefore, it reduces to proving $f\circ s_k\in \GG$ for large $k$.
							
							We assume $\supp(f\circ s)\subset [-N,N]$ for any above $k$. Fix $k$ large enough such that $s_k(N)<s_k(\infty)$ and $s_k(-N)>s_k(-\infty)$, i.e. $\d s_k$ assigns positive masses on both $(-\infty,-N)$ and $(N,\infty)$. We denote $M=\inf\{x:s_k(x)=s_k(\infty)\}$, which is finite and greater than $N$. 
							
							In the following, we shall show that $\GG$ contains a special class of ladder-like functions. Roughly speaking, the functions have a constant stretch beyond the maximum and their `left arms' are $h^l_n$ or $h^r_n$. Before the precise statement, we introduce some notation. Let $h$ and $p$ be two ladder-like functions. We say $h$ and $p$ are compatible if 
							$\norm{h}_\infty=\norm{p}_\infty$, $\supp(h)\cap \supp(p)=\emptyset$. For compatible $h$ and $p$, assuming $T^h<T^p$ and $\supp(p)=[a,b]$, the functions $A^l(h,p)$ and $A^r(h,p)$ are defined as:
							\begin{align*}
								A^l(h,p)(x)&=\begin{cases}
									h^l(x),\ & x\in(-\infty,T^p];\\
									p(x),\ & x\in (T^p,\infty),
								\end{cases}\\
								A^r(h,p)(x)&=\begin{cases}
									h^r(x),\ & x\in(-\infty,a];\\
									\norm{p}_\infty-p(x),\ & x\in (a,T^p];\\
									0,\ &(T^p,\infty).
								\end{cases}
							\end{align*}
							Namely, $A^l(h,p)$ and $A^r(h,p)$ are ladder-like functions. The `arms' of $A^l(h,p)$ consist of the `left arm' of $h$ and the `right arm' of $p$, while the `arms' of $A^r(h,p)$ consist of the reflections of the `right arm' of $h$ and the `left arm' of $p$. Let
							\[E:=\big\{p\in C_c: p\text{ is ladder-like, }\supp(p)\subset [M+1,M+2]\big\}.\]
							We claim that for any $n\ge1$, there exists $p_n\in E$ such that $p_n$ and $h_n$ are compatible and both $A^l_n:=A^l(h_n,p_n)$ and $A^r_n=A^r(h_n,p_n)$
							are in $\GG$. If it is proved, then we have
							\[q_k:=\sum_{n=1}^k 2^{-n} \dfrac{A^l_n +A^r_n}{\norm{h^l_n}_{\mathcal{E}}+\norm{h^r_n}_{\mathcal{E}}+\norm{h_n}_\infty}\in \GG.\]
							Note that $q_k=s_k$ on $(-\infty,M+1]$. On $(M+1,\infty)$, $s_k$ is constant while $q_k$ decreases to $0$ and stays at $0$ afterwards. This implies that $f\circ s_k\in \GG$. Indeed, since $f\circ q_k$ is a contraction of $\norm{f^\prime}_\infty q_k$, it is also in $\GG$. Moreover, it holds that 
							\begin{align*}
								f\circ q_k=f\circ s_k\text{ on }(-\infty,M)\supset\supp(f\circ s_k),\quad f\circ q_k=0\text{ on }(M,M+1).
							\end{align*}
							In view of this,
							we take  $h\in \GG \cap C_c$ such that $h(x)\in (0,1)$, $h(x)=1$ on $\supp(f\circ s_k)$, and $h(x)=0$ on $(M+1,\infty)$. It follows that $f\circ s_k=h\cdot (f\circ q_k)\in \GG$.
							
							Now we prove the claim. 
							Let $p_n$ be a function in  $E\cap\GG$ with $\norm{p_n}_\infty=\norm{h_n}_\infty$. The regularity of $\GG$ and the ladder decomposition ensure the existence of such functions. Since $A^l_n, A^r_n\in H^{\alpha/2}$, it suffices to prove for any $\varepsilon>0$, 
							\begin{enumerate}[(1)]
								\item\label{GG1} $A^l_n\wedge (\norm{h_n}_\infty-\varepsilon)\in \GG$.
								\item\label{GG2} $(A^r_n)_\varepsilon\in \GG$, where $(\cdot)_\varepsilon$ is the $\varepsilon$-contraction defined in Section \ref{presection}.
							\end{enumerate}
							
							For \eqref{GG1}, it is easy to see that there exists $g\in \GG\cap C_c$ such that $\supp(g)\subset (T^{h_n},T^{p_n})$ and $g+h_n+p_n\ge \norm{h}_\infty-\varepsilon$ on $(T^{h_n},T^{p_n})$. So 
							\[A^l_n\wedge (\norm{h_n}_\infty-\varepsilon)=(g+h_n+p_n)\wedge (\norm{h_n}_\infty-\varepsilon)\in \GG.\]
							
							For \eqref{GG2}, we take $g\in \GG\cap C_c$ such that $g=\norm{h_n}_\infty$ on $[T^{h_n},T^{p_n}]$ and $g\ge h_n+p_n$. Then $g-h_n-p_n=A^r_n$ on $\supp(A^r_n)$. Note that $g-h_n-p_n=0$ on $\partial \supp(A^r_n)$. It follows that $(g-h_n-p_n)_\varepsilon=0$ on a neighbourhood $\partial \supp(A^r_n)$. So we can `discard' the part of $(g-h_n-p_n)_\varepsilon$ in $\supp(A^r_n)^c$ by multiplying a function in $\GG\cap C_c$. That deduces \eqref{GG2} and completes the whole proof.

							\section{Proof of Theorem \ref{char2}}\label{charsection}
							We have seen that for $1< \alpha\le 2$ and $s\in \bS^{\alpha/2}_\loc$, $\overline{\cF_s}$ is a regular subspace of $H^{\alpha/2}$. In this section, we shall further explore the problem: for what kind of $s$, $\overline{\cF_s}$ is proper? To begin with, we recall some known results of the extended Dirichlet space of $H^{\alpha/2}$.
							
							For $0<\alpha\le 2$, let $H^{\alpha/2}_e$ be the extended Dirichlet space of $H^{\alpha/2}$. When $0<\alpha<1$, $H^{\alpha/2}_e$ is Hilbert due to the transience of $H^{\alpha/2}$. When $1\le \alpha<2$, it has the following characterization (Cf. \cite[\S6.5]{chen2012symmetric}):
							\begin{align*}
								\left\{
								\begin{aligned}
									H^{\alpha/2}_e&=\left\{f \in L^{2}_{loc}(\IR):\int_{\mathbb{R}^2}\dfrac{(f(x)-f(y))^2}{|x-y|^{1+\alpha}}\d x\d y<\infty \right\};\\
									\EE(f,f)&=\int_{\mathbb{R}^2}\dfrac{(f(x)-f(y))^2}{|x-y|^{1+\alpha}}\d x\d y. 
								\end{aligned}\right.
							\end{align*}
							The quotient space of $H^{\alpha/2}_e$
							by the space of constant functions $\mathcal{N}$, denoted by $\mathring{H}^{\alpha/2}_e=H^{\alpha/2}_e/\mathcal{N}$, is a Hilbert space.
							
							In the following, we assume $\alpha\in (1,2)$ and prove Theorem \ref{char2}, where we exclude the case $\alpha=2$ which has been known. Recall that we denote $\alpha^*=2-\alpha$ and assume all the quasi notions are relative to $(H^{\alpha^*/2},\EE^{\alpha^*/2})$ henceforth.

								For any $g\in H^{\alpha/2}$, we define $Dg$ to be the derivative of $g$ in the sense of distribution, i.e. for any $\phi\in C^\infty_c$, $Dg(\phi)=-\int_{\IR} g\phi^\prime\d x$. Then we have
								\begin{align}\label{ddg}
									\abs{Dg(\phi)}=\abs{\big(g,\phi^\prime\big)_{L^2(\IR)}}=\abs{\big((-\Delta)^{\alpha/4}g,(-\Delta)^{\alpha^*/4}\phi\big)_{L^2(\IR)}}\le \norm{g}_{\EE^{\alpha/2}}\norm{\phi}_{\EE^{\alpha^*/2}},        
								\end{align}
								where $\Delta$ is the Laplacian and it holds that $\norm{g}_{\EE^{\alpha/2}}=\norm{(-\Delta)^{\alpha/4}g}_{L^2(\IR)}$.  It follows from \eqref{ddg} that $Dg$ can be generalized to an element in $\big(H^{\alpha^*/2}_{e}\big)^*$ (the dual space of $H^{\alpha^*/2}_{e}$), which we still denote by $Dg$. Since $(-\Delta)^{\alpha^*/4}$ is self-adjoint and injective, we have $(-\Delta)^{\alpha^*/4}(C^\infty_c)$ is dense in $L^2(\IR)$. Hence  the operator norm of $Dg$ equals $\norm{g}_{\EE^{\alpha/2}}$. Namely, $D$ is an isometry from $H^{\alpha/2}$ to $\big(H^{\alpha^*/2}_{e}\big)^*$, which can be further extended to an isometry from $\mathring{H}^{\alpha/2}_e$ to $\big(H^{\alpha^*/2}_{e}\big)^*$.  Now we take $g=f\circ s\in \cF_s$. Since $f\circ s$ is of BV, using the formula for integration by parts, we have
								\begin{align}\label{integral}
									D(f\circ s)(\phi)=\int_{\IR} \phi\,\d (f\circ s),\text{ for all $\phi\in C^\infty_c$}.
								\end{align}
								Combining with \eqref{ddg}, we have the measure $\d(f\circ s)$ is of finite energy integral relative to $\big(H^{\alpha^*/2},\EE^{\alpha^*/2}\big)$, which simply leads to the smoothness of $\d s$. In particular, it doesn't charge sets of zero capacity.
								Based on this, we can generalize \eqref{integral} as follows. Recall that $\widetilde{\phi}$ represents the quasi continuous version of $\phi$. 
								\begin{lemma}\label{qelem}
									For any $f\circ s\in \cF_s$,
									\begin{align}\label{generalize}
										D(f\circ s)(\phi)=\int_{\IR} \widetilde{\phi}\,\d (f\circ s),\text{ for all $\phi\in H^{\alpha^*/2}_e$}.
									\end{align}
								\end{lemma}
								\begin{proof}
									For any bounded $\phi\in H^{\alpha^*/2}_e$, we take a sequence of uniformly bounded functions $\{\phi_n\}\subset C^\infty_c$ such that $\phi_n$ converges to $\phi$ in $\EE^{\alpha^*/2}$ and $\phi_n$ converges q.e. to $\widetilde{\phi}$. Then by substituting $\phi_n$ in place of $\phi$ in \eqref{integral} and letting $n\uparrow\infty$, we reach \eqref{generalize} for bounded $\phi$. The further genralization to any $\phi\in H^{\alpha^*/2}_e$ is obtained by truncation.
								\end{proof}
								
								
								
							Now let us prove \eqref{proper4}$\Rightarrow$\eqref{proper1}. Suppose $f\in H^{\alpha^*/2}$ is nonzero and $\langle\widetilde{f},\d s\rangle=0$. Then by slightly extending \eqref{ddg}, we have for any $\phi\circ s\in \cF_s$, 
							\begin{align*}
								\abs{\langle\widetilde{f},\d(\phi \circ s)\rangle}=\abs{\big((-\Delta)^{\alpha/4}\phi\circ s  ,(-\Delta)^{\alpha^*/4}f \big)_{L^2(\IR)}}= \abs{\big(\phi\circ s  ,(-\Delta)^{1/2-\alpha/2}f \big)_{\mathcal{E}^{\alpha/2}}} =0,
							\end{align*}
							which deduces $\cF_s$ is not dense in $H^{\alpha/2}$, since otherwise we can approximate $(-\Delta)^{1/2-\alpha/2}f\in H^{\alpha/2}$ by elements in $\cF_s$, which implies  \[\sup_{\phi\circ s\in \cF_s}\abs{\big(\phi\circ s  ,(-\Delta)^{1/2-\alpha/2}f \big)_{\mathcal{E}^{\alpha/2}}}>0.\]
							
							
							
							

							For the proof of \eqref{proper1}$\Rightarrow$\eqref{proper4}, we need another lemma. Recall that $\mathcal{N}$ is the space of constant functions.
							\begin{lemma}\label{extended-close}
								Suppose $\overline{\FF_s}$ is a proper regular subspace of $H^{\alpha/2}$. Let $\GG$ be the extended Dirichlet space of $\overline{\FF_s}$ and $\mathring{\GG}:=(\GG+\mathcal{N})/\mathcal{N}$. Then $\mathring{\GG}$ is a proper closed subspace of $\mathring{H}^{\alpha/2}_e$.
							\end{lemma}
							\begin{proof}
								It is easy to see that $\overline{\cF_s}$ is irreducible. So it is either recurrent or transient. The proof proceeds separately in these two cases.
								
								If $\overline{\cF_s}$ is recurrent, then $\mathcal{N}\subset \GG$, which implies $\mathring{\GG}$ is proper, since otherwise we have $\GG=H^{\alpha/2}_e$ and \[\overline{\cF_s}=\GG\cap L^2(\IR)=H^{\alpha/2}_e\cap L^2(\IR)=H^{\alpha/2}.\] 
								Now let us show $\mathring{\GG}$ is closed. Recall that $\mathring{H}^{\alpha/2}_e$ is Hilbert. It suffices to prove that if $\{f_n\}\subset \GG$ converges in $\EE$-norm to $f\in H^{\alpha/2}_e$, then $f\in \GG$. Note that the convergence is equivalent to that $g_n(x,y):=f_n(x)-f_n(y)$ converges to $g(x,y):=f(x)-f(y)$ in $L^2(\mu)$, where $\mu:=\abs{x-y}^{-(1+\alpha)}\d x\d y$. It follows that $g_n(x,y)\rightarrow g_n(x,y)$ for a.e. $(x,y)$ (or a subsequence at least). We can find some $x_0$ such that $g_n(x_0,y)\rightarrow g(x_0,y)$ for a.e. $y$. Since $\mathcal{N}\subset \GG$,  without loss of generality, we assume $f_n(x_0)=f(x_0)$ for all $n$. Then $f_n(x)\rightarrow f(x)$ a.e.. 
								Finally, we consider any approximating sequence $\big\{f^{(m)}_n:m\ge 1\big\}\subset \overline{\cF_s}$ of $f_n$. We can find a sequence $\big\{f^{(m_n)}_n:n\ge 1\big\}$ and $x_1\in \IR$ such that $f^{(m_n)}_n\rightarrow f$ in $\EE$-norm and $f^{(m_n)}_n(x_1)\rightarrow f(x_1)$. Using the same arguments as before, we have $f^{(m_n)}_n(x)\rightarrow f(x)$ a.e.. Hence $f\in \GG$. 
								
								If $\overline{\cF_s}$ is transient, then $\GG$ contains no nonzero constant functions and $\GG$,  $\mathring{\GG}$ are both Hilbert. It remains to show $\mathring{\GG}$ is proper. The proof goes by contradiction. Suppose $\mathring{\GG}=\mathring{H}^{\alpha/2}_e$. Then for each non-negative $f\in C^\infty_c$, there exists a unique constant $c=c(f)$ such that $f+c\in \GG$. We shall verify that any possible case leads to a contradiction.
								\begin{itemize}
									\item If $c=0$ for all $f$, then $\overline{\cF_s}=\GG\cap L^2(\IR)$ contains all non-negative functions in $C^\infty_c$. So $\overline{\cF_s}=H^{\alpha/2}$, which contradicts.
									\item If $c>0$ for some $f$, then $c=c\wedge (c+f)\in \GG$, which contradicts the transience.
									\item If $c<0$ for some $f$, we take $g\in \overline{\cF_s}\cap C_c$ such that $g+f\le 0$. Then $c=c\vee (g+f+c)\in\GG$. Again, it contradicts the transience.
								\end{itemize}
								Therefore, it follows that $\mathring{\GG}$ is proper.
							\end{proof}

						From the Lemma \ref{extended-close}, there exists a non-zero $g\in \mathring{H}^{\alpha/2}_e$ such that $g\perp \FF_s$ in $\EE$-norm. Since $D$ is isometric, it also holds that
						\[Dg\perp \{\d(f\circ s):f\circ s\in \cF_s\}\text{ in }\big(H^{\alpha^*/2}_e\big)^*.\]
						In particular, since $g$ is non-constant, $Dg$ is also nonzero. 
						Thus, by Hahn-Banach theorem and Lemma \ref{qelem}, there exists a nonzero $h\in H^{\alpha^*/2}_e$ such that $\int_{\IR} \widetilde{h}\,d(f\circ s)=0$ for all $f\circ s\in \FF_s$. Or equivalently, $\langle \widetilde{h},\d s\rangle=0$. We further take any $\phi\in C^\infty_c$ and set $q:=(h\wedge 1)\cdot \phi$. Then $q\in H^{\alpha^*/2}_e\cap L^2(\IR)= H^{\alpha^*/2}$. It is nonzero and satisfies  $\langle \widetilde{q},\d s\rangle=0$.
						
						Next, we shall further prove the equivalence of \eqref{proper2}$\sim$\eqref{proper4}.
						For this, we first present an equivalent condition for $\Cap^{\alpha^*/2}$ to be  concentrated on a Borel set $E$. 
						
						\begin{theorem}[{\cite[Theorem 11.3.2]{adams1999function}}]\label{adams}
							Let $0<\alpha^*\le1$ and $E$ be a Borel set. The following are equivalent:
							\begin{enumerate}[(1)]
								\item\label{adams1} $\Cap^{\alpha^*/2}$ is not concentrated on $E$;
								\item\label{adams2} there exists a nonzero $f\in H^{\alpha^*/2}$ such that $\widetilde{f}(x)=0$ q.e. on $E$.
							\end{enumerate}
						\end{theorem}
						\begin{proof}
							This theorem is proved in \cite[Theorem 11.3.2]{adams1999function}, where it is stated in terms of the Bessel $(\alpha^*,2)$-capacity which is equivalent to $\Cap^{\alpha^*/2}$.
							For the sake of completeness, we supply here an alternative proof with a probabilistic method. 
							We write $\Cap=\Cap^{\alpha^*/2}$ in this proof.
							
							\eqref{adams1}$\Rightarrow$\eqref{adams2}: Suppose there exists an open set $U$ such that $\Cap(E\cap U)<\Cap(U)$. Let $u\in H^{\alpha^*/2}$ be the capacitary function for $E\cap U$, that is characterized by: $$\widetilde{u}(x)= 1 \text{ q.e. on } E\cap U,\text{ and }\mathcal{E}_1(u,u)=\Cap(E \cap U).$$  It holds that $u(x)<1$ on a subset of $U\setminus E$ of positive Lebesgue measure, since otherwise $\Cap(U \setminus E)=\Cap(U)$. Then there exists a compact set $K\subset U$ such that 
							$K\cap \{u<1\}$ has positive Lebesgue measure. 
							We take any $\varphi\in C_c^\infty(\mathbb{R})$ with $\varphi(x)=1$ on $K$ and $\supp(\phi)\in U$. Then $f:=\varphi\cdot(1-u)=\varphi-\varphi\cdot u\in H^{\alpha^*/2}$ and it is nonzero. Moreover, $\widetilde{f}=\varphi\cdot(1-\widetilde{u})$  vanishes q.e. on $E$.

							\eqref{adams2}$\Rightarrow$\eqref{adams1}:
							Suppose $\Cap(E\cap U)=\Cap(U)$ for an open set $U$. Denote by $\tau_{E\cap U}$ and $\tau_U$ the hitting times of $E\cap U$ and $U$ by a symmetric stable process with parameter $\alpha^*$ respectively. Then $\mathbb{P}_x(e^{-\tau_{E\cap U}})$ (resp. $\mathbb{P}_x(e^{-\tau_U})$) is a quasi-continuous version of the capacitary function for $E\cap U$ (resp. $U$).
							We have $\mathbb{P}_x(e^{-\tau_{E\cap U}})=\mathbb{P}_x(e^{-\tau_U})$ for q.e. $x$, since otherwise  $\Cap(E\cap U) < \Cap(U)$. It follows that for q.e. $x\in U$ (indeed, for every $x$ considering the transition probability has a density), $x$ is a regular point of $E\cap U$. Therefore, if $f \in H^{\alpha^*/2}$ with $\widetilde{f}(x)=0$ q.e. on $E\cap U$, it also vanishes q.e. at the regular points of $E\cap U$ due to the quasi continuity. So $\widetilde{f}(x)=0$ q.e. on $U$. That immediately leads to the conclusion. 
						\end{proof}	
						Back to Theorem \ref{char2}, 
						\eqref{proper2}$\Leftrightarrow$\eqref{proper4} is immediate from Theorem \ref{adams} and the fact that for any non-negative quasi continuous function $\widetilde{f}\in H^{\alpha/2}$, $\langle \widetilde{f},\d s\rangle=0$ if and only if $\widetilde{f}=0$ q.e. on the quasi support of $\d s$. \eqref{proper2}$\Rightarrow$\eqref{proper3} is obtained by noting that the  quasi support of $\d s$ is a quasi closed set, whose capacity can be approximated by a descending sequence of open sets.
						For \eqref{proper3}$\Rightarrow$\eqref{proper4}, suppose $E$ satisfies the conditions in \eqref{proper3}. By Theorem \ref{adams}, there exists a nonzero $f\in H^{\alpha^*/2}$ such that $\widetilde{f}=0$ q.e. on $E$. Then 
						\[\langle \widetilde{f},\d s\rangle=\langle \widetilde{f}1_E,\d s\rangle+\langle \widetilde{f}1_{E^c},\d s\rangle=0.\]
						Namely, $f$ satisfies the conditions in \eqref{proper4}. That concludes the theorem.
						\begin{remark}\label{rmk}
							Almost all the conclusions of Theorem \ref{char1} and \ref{char2} can actually be generalized to the case $\alpha\in(0,1]$. The only exception is the second part of Theorem \ref{char1} may not hold in the case $\alpha=1$. The proofs are basically the same as that in Section \ref{charsect} and Section \ref{charsection}. We mention that for the proof of `\eqref{proper1}$\Rightarrow$\eqref{proper4}' in Theorem \ref{char2}, one needs to occaasionally consider $\mathring{H}^{\alpha^*/2}_e$ instead of $H^{\alpha^*/2}_e$. The following are some consequences of this generalization.
							\begin{itemize}
								\item For $\alpha\in(0,1)$,
								since all the functions in $H^{\alpha^*/2}$ is continuous, using \eqref{proper1}$\Leftrightarrow$\eqref{proper4} in Theorem \ref{char2}, we get $H^{\alpha/2}$ has no proper subspace, which is another proof the `only if' part of Theorem \ref{mainthm}. However, this proof relies on the duality of $H^{\alpha/2}$ and $H^{\alpha^*/2}$; while the proof in Section \ref{case1} can actually be generalized to the L\'evy processes with finite variation (see Section \ref{generalsec} for details);
								\item For $\alpha=1$, the corresponding versions of the first part of Theorem \ref{char1} and Theorem \ref{char2} allow us to construct the special proper regular subspaces of $H^{1/2}$, which will be presented in the next section.
							\end{itemize}
						\end{remark}


						\section{Proof of the `if' part of Theorem \ref{mainthm}}\label{ifpart}
						In this section, we will construct special proper regular subspaces to conclude the `if' part of Theorem \ref{mainthm}.

						Due to Theorem \ref{char2} and Remark \ref{rmk}, it suffices to find $s\in \bS^1_\loc$ satisfying \eqref{proper3} in Theorem \ref{char2}. We focus on the following type of $s$. Let $G$ be a dense open subset of $\IR$. Define
						\begin{align}\label{sss}
							s(x):=\int_{0}^{x}1_G(y)\d y,\ x\in\IR.
						\end{align}
						Then $s\in \bS^1_\loc$ and $\d s$ is concentrated on $G$.
						It is known (Cf. \cite[Chapter 5]{adams1999function}) that for any interval $I_r:=(-r,r)$, there exists $C,A,a>1$ such that 
						\begin{align*}
							\begin{cases}
								C^{-1}r^{\alpha-1}\le\Cap^{\alpha^*/2}(I_r)\le Cr^{\alpha-1},&\text{ for }1<\alpha<2;\\
								A^{-1}(\log(a/r))^{-1}\le  \Cap^{\alpha^*/2}(I_r)\le A(\log(a/r))^{-1},&\text{ for }\alpha=1.
							\end{cases}
						\end{align*}
						Hence for any $\alpha\in(1,2)$ (resp. $\alpha=1$), we can take $r_i$, $i\ge 1$, small enough, such that
						\[\sum_i C r_i^{\alpha-1} < \Cap^{\alpha^*/2}(I_1)\, (resp. \,\sum_i A(\log(a/r_i))^{-1}  < \Cap^{1/2}(I_1)). \]
						Define $G=G(\alpha)=[-1,1]^c\cup \bigcup_{i\ge 1} (x_i+I_{r_i})$, where $\{x_i\}_{i \ge 1}$ is a dense subset of $I_1$. Then $G$ is dense and
						$\Cap^{\alpha^*/2}(G\cap I_1)<\Cap^{\alpha^*/2}(I_1)$. It implies that the corresponding $s$ satisfies \eqref{proper3} in Theorem \ref{char2}. We have thus proved the theorem.

						\section{Generalizations to L\'evy processes}\label{generalsec}
						
						The characteristic exponent of a symmetric L\'evy process on $\IR$ is 
						$$\psi(\xi)=\frac{1}{2}\sigma \xi^2+\int_{\mathbb{R}}(1-\cos(\xi x))\nu(\d x),$$
						where $\sigma\geq 0$, $\nu$ is a symmetric Radon measure on $\IR\setminus\{0\}$ with integrability $$\int_{\mathbb{R}}(1\wedge x^2 )\nu(\d x)<\infty.$$
						It is known that a L\'evy process has finite variation, i.e.
						its sample path has bounded variation over any finite interval a.s., if and only if $\sigma=0$ and $\int_{\mathbb{R}}(1\wedge |x|)\nu(\d x)<\infty$.
						The latter is equivalent to for any $\lambda>0$,
						$$\int_{\IR} (\lambda \wedge |x|)\nu(\d x)<\infty.$$
						
						Assume $\sigma=0$ in the first part of this section and then the Dirichlet form of the process can be written as
						\begin{align*}
							\left\{\begin{aligned}
								&\mathcal{F}=\left\{f\in L^2(\mathbb{R}):\int_{\mathbb{R}}\big|\widehat{f}(\xi)\big|^2 \psi(\xi)d\xi<\infty\right\};\\
								&\mathcal{E}(f,g)=\frac{1}{2}\int_{\mathbb{R}^2}(f(x)-f(y))(g(x)-g(y))\nu(\d x-y)\d y=\int_{\mathbb{R}}\widehat{f}(\xi) \bar{\widehat{g}}(\xi)\psi(\xi)\d\xi.\end{aligned}\right.
						\end{align*}
						In particular, $C_c^{\infty}\cap \FF$ is a core of $(\EE,\FF)$.
						
						The following lemma, obtained by direct computation,
						generalizes Lemma \ref{lemma11}.
						\begin{lemma}
							For the indicator function $f=1_{[a,b]}$ with $a<b$, it holds that 
							$$\int_{\IR^2} (f(x)-f(y))^2\nu(\d x-y)\d y
							=2\int_{\IR} ((b-a)\wedge |x|)\nu(\d x).$$
							Hence $1_{[a,b]}\in \FF$ if and only if $\displaystyle{\int_{\IR} (1\wedge |x|)\nu(dx)<\infty}$.
						\end{lemma}
						
						Then the following theorem is quite anticipated.
						
						\begin{theorem}
							Any 1-dim symmetric L\'evy process with finite variation has no proper regular subspaces.
						\end{theorem}
						\begin{proof}
							The proof of the `only if' part of Theorem \ref{mainthm} also works in this case, if we can check the analogs of Lemma \ref{S0} and \ref{hn} for the L\'evy process with finite variation. Therefore, we only need to verify the analogs of Lemma \ref{S0} and \ref{hn}, i.e.
							\begin{enumerate}[(i)]
								\item\label{levy1} $S_0$ is dense in $\mathcal{F}$;
								\item\label{levy2} with the same notations as in Section \ref{case1},  $h_\rho\in\mathcal{F}$ for $\rho>0$, and $\{\EE(h_\rho,h_{\rho}):\rho>0\}$ is bounded.
							\end{enumerate}
							
							For the proof of \eqref{levy1}, because it is known that $C_c^\infty(\IR)\cap\FF$ is
							dense in $\FF$ in $\EE_1$-norm, we need only to prove
							that $S_0$ is dense in $C^\infty_c\cap\FF$ in $\EE_1$-norm. 
							Take $f\in C_c^\infty\cap\FF$ with supp$(f)\subset (-N,N)$ and $|f'|\le c$ and $f_n$ is defined as in Lemma \ref{S0}. Write $g_n=f-f_n$, then $|g_n(x)|\le c/2^n$ for all
							$x\in\IR$. It suffices to show that
							$\EE(g_n,g_n)\longrightarrow 0$.
							
							The idea is similar to the proof of
							Lemma \ref{S0} with a slight difference.
							\begin{align*}
								I:=\int_{\mathbb{R}^2}((g_n)(x)-(g_n)(y))^2 \nu(\d x-y)\d y
							\end{align*}
							and we can also divide it into three parts.
							
							The first part is similar to the Lemma \ref{S0} and we have
							\begin{align*}
								I^\prime_1&=2\int_{[-N,N]\times (\mathbb{R}\backslash [-N,N]) }(g_n(x)-g_n(y))^2\nu(\d x-y)\d y\\
								&\le 2\int_{[-N,N]\times (\mathbb{R}\backslash [-N,N]) }{\Big(c\frac{1}{2^n}\Big)^2}\nu(\d x-y)\d y=\dfrac{2c^2}{ 2^{2n}}\int_{\IR}(|x|\wedge 2N)\nu(\d x).
							\end{align*}
							The second part is the integral on $[-N,N]\times[-N,N]\setminus D_n$
							where $$D_n=\{|x-y|>1/2^n,|x|<N,|y|<N\},$$
							and we have\begin{align*}
								I^\prime_2&= 
								\int_{[-N,N]\times[-N,N]\setminus D_n}
								(g_n(x)-g_n(y))^2\nu(dx-y)dy\\
							&\le (2c/2^n)^2\int_{\frac{1}{2^n}\le |x-y|\le 2N, |y|\le N
							}\nu(\d x-y)dy\le \dfrac{4c^2}{2^{2n}}\int_{ \frac{1}{2^n}\le |x|\le 2N}2N\nu(dx)\\
							&=
							\dfrac{8Nc^2}{2^{n}}\int_{\frac{1}{2^n}\le |x|\le 2 N}\dfrac{1}{2^n}\nu(dx)\le
							\dfrac{8Nc^2}{2^{n}}\int_{-2N}^{2N}|x|\nu(dx).\end{align*}
						The third part is the integral on $D_n$ which
						is the union $\displaystyle{\bigcup_{-N2^n\le i<N2^n}\Big[\dfrac{i}{2^n},\dfrac{i+1}{2^n}\Big]^2}$ and its
						complement. It is essential for the former
						that $$|g_n(x)
						-g_n(y)|=|f(x)-f(y)|\le c|x-y|\ \text{for}\  
						x,y\in \Big[\dfrac{i}{2^n},\dfrac{i+1}{2^n}\Big],$$ and hence it follows that
						\begin{align*}
							I^\prime_3
							&\le \sum_{-N2^n\le i<N2^n}
							\int_{x,y\in[\frac{i}{2^n},\frac{i+1}{2^n}]}c^2(x-y)^2\nu(dx-y)
							dy\\
							&\ \ +2\sum_{-N2^n\le i<N2^n}
							\int_{x\in[\frac{i+1}{2^n},\frac{i+2}{2^n}],y\in [\frac{i}{2^n},\frac{i+1}{2^n}],
								x<y+\frac{1}{2^n}}(2c/2^n)^2\nu(dx-y)dy\\
							&\le 2N2^n\left(\dfrac{c^2}{2^{n}}\int_{|x|\le \frac{1}{2^n}} x^2\nu(dx)+
							\dfrac{4c^2}{2^{2n}}\int_{|x|\le \frac{1}{2^n}} |x|\nu(dx)\right).
						\end{align*}
						By the assumption, this concludes that $I=I'_1+ I'_2+ I'_3$ converges to $0$.
						
						Finally, we prove \eqref{levy2}. Similarly, we shall compute the $\EE$-norm of
						$h_\rho$.
							It follows from the inequality \eqref{keyinequality}
							that $$|\hat h_{\rho}(\xi)|\le 2(|\xi|^{-1}\wedge 1).$$
							Combining $0\le 1-\cos y\le 2(1\wedge y^2)$, we have 
							\begin{align*} \EE(h_{\rho},h_\rho)=&
								\int_{\mathbb{R}}|\hat{h}_\rho(\xi)|^2|\psi(\xi)|\d\xi\\
								\le & 4\int_{\mathbb{R}^2}(\xi^{-2}\wedge1)(1-\cos(\xi x))\d\xi \nu(\d x)\\
								\leq
								&8\int_{\IR} \nu(dx)\int_{\IR}(\xi^{-2}\wedge1)(1\wedge(\xi x)^2)\d\xi\\
								=&8\int_{\IR}\left((2-\dfrac{2}{3|x|})1_{\{|x|\ge 1\}}
								+(2-\dfrac{2}{3}|x|)|x|1_{\{|x|<1\}}\right)\nu(\d x)\\
								\le &16\int_{\IR}
							(1\wedge |x|)\nu(\d x), 
						\end{align*}
					which is a finite number independent of $\rho$.
					Now the other parts of the proof are exactly the same. We mention that since \eqref{levy2} already shows $h_\rho\in \cF$, so we do not need the `Step 1' part in the proof of Lemma \ref{hardy}. 
				\end{proof}
				
				We next generalize the other part of Theorem \ref{mainthm}. In this part, we assume $\sigma\ge 0$.
				
				\begin{corollary}
					If there exist some positive constants $c$ and $M$ such that $\psi(\xi)\geq c\abs{\xi}^\alpha$,$\forall\, \abs{\xi}\geq M$ for some $\alpha\in [1,2]$, then the L\'evy process with characteristic exponent $\psi$ has proper regular subspaces. In particular, when the parameter of quadratic term $\sigma>0$, the corresponding process has proper regular subspaces.
				\end{corollary}
				
				\begin{proof}
					We again consider the scale function in the form of \eqref{sss}. It is easy to see that $\overline{\cF_s}^{\EE_1}$ is a regular subspace of $\cF$. 
					Considering $\overline{\cF_s}^{\alpha/2} \subset \overline{\cF_s}^{\EE_1}$, we can check that \eqref{proper2}$\Leftrightarrow$\eqref{proper3}$\Leftrightarrow$\eqref{proper4}$\Rightarrow$\eqref{proper1} still holds in this case with $\alpha$ given by the statement of the corollary. Then the conclusion follows from the proof in Section \ref{ifpart}. 
				\end{proof}
				
				\bibliographystyle{amsalpha}
				\bibliography{regular}
				\nocite{*}
			\end{document}